\documentclass{amsproc}
\usepackage{graphicx} 
\usepackage[colorlinks=true, allcolors=blue]{hyperref}
\usepackage{caption}
\usepackage{subcaption}
\usepackage{epstopdf} 
\usepackage{bbm}
\usepackage{overpic}
\usepackage{rotating}
\usepackage{xcolor}
\newcommand{\D}{\text{D}}

\DeclareGraphicsExtensions{.png,.pdf,.jpg}

\newtheorem{theorem}{Theorem}[section]

\newtheorem{prop}[theorem]{Proposition}

\theoremstyle{definition}
\newtheorem{definition}[theorem]{Definition}

\theoremstyle{remark}
\newtheorem{remark}[theorem]{Remark}
\numberwithin{equation}{section}

\DeclareMathOperator{\id}{id}
\newcommand{\rmd}{\mathrm{d}}

\begin{document}

\title{Noise-induced instabilities in a stochastic Brusselator}
\author{Maximilian Engel}
\address{Freie Universit\"at Berlin}
\email{g.olicon.mendez@fu-berlin.de}

\author{Guillermo Olic\'on-M\'endez}
\address{Freie Universit\"at Berlin}
\email{maximilian.engel@fu-berlin.de}



\date{}

\begin{abstract}
We consider a stochastic version of the so-called Brusselator -  a mathematical model for a two-dimensional chemical reaction network - in which one of its parameters is assumed to vary randomly. 
It has been suggested via numerical explorations that the system exhibits noise-induced synchronization when time goes to infinity. Complementing this perspective, in this work we explore some of its finite-time features from a random dynamical systems perspective. In particular, we focus on the deviations that orbits of neighboring initial conditions exhibit under the influence of the same noise realization. For this, we explore its local instabilities via \textit{finite-time Lyapunov exponents}.
Furthermore, we present the stochastic Brusselator as a fast-slow system in the case that one of the parameters is much larger than the other one. In this framework, an apparent mechanism for generating the stochastic instabilities is revealed, being associated to the transition between the slow and fast regimes.
\end{abstract}

\maketitle

\section{Introduction}
\label{SEC:Intro}
In the past few decades there has been an increased interest in the effects of noisy perturbations on deterministic dynamical systems and their bifurcations, i.e.~critical parameter-dependent changes of qualitative behaviour in the system. In many cases, such a change of behaviour mirroring the deterministic bifurcation can be observed in the density of the respective stationary distribution (if it exists). The change of its shape, typically observed in the local maxima of such density functions, has been called a \textit{P-bifurcation}, standing for \textit{phenomenological bifurcation} \cite{arnold1998}. 

Another direction has been addressing the qualitative change in the structure of the objects which remain invariant under the random dynamics. In this context, a seminal insight was obtained at the hand of the pitchfork bifurcation with additive noise \cite{CrauelFlandoli98}, stating that for any such perturbation of the pitchfork normal form the system admits a unique random equilibrium independently from the parameter value,
whereas the deterministic system exhibits a bifurcation from one to two locally attracting equilibria. Even in the bistable scenario, when the system is perturbed with additive noise the unique random equilibrium is globally attracting, from which it follows that any two trajectories of the system approach each other as time grows. This phenomenon is commonly known as \emph{noise-induced synchronization}.
However, it has been argued that the random bifurcation persists in a more subtle manner, being detected by changes of finite-time stability and the dichotomy spectrum \cite{Callawayetal}. This point of view, measuring a random bifurcation in particular by means of \emph{finite-time Lyapunov exponents} (FTLEs), has been extended to Hopf bifurcations in normal form \cite{Doanetal} and the infinite-dimensional situation recently \cite{BlumenthalEngelNeamtu}. This perspective puts emphasis on the transient dynamics rather than on the asymptotic behaviour of the orbits, which becomes of particular relevance in the applied sciences.

Despite the apparent asymptotic stabilizing effect that noisy perturbations often have on deterministic systems, there has been recent progress on showing \textit{noise-induced chaos} for particular classes of stochastically perturbed oscillators \cite{EngelLambRasmussen1}, implying a stochastic bifurcation from synchronization to chaotic attractors. 
Building up on numerical investigations  \cite{LinYoung2008} and \emph{Furstenberg-Khasminskii} formulas for Lyapunov exponents \cite{ImkellerLederer1999, ImkellerLederer2001}, 
it was shown lately \cite{ChemnitzEngel2022} that a positive Lyapunov exponent can be obtained for appropriate parameter combinations in a model of Hopf bifurcations with shear terms and additive noise.

In this work we turn our attention to a stochastic version of the Brusselator. This system is an example of a chemical reaction network whose reaction rate equation exhibits a Hopf bifurcation, i.e.~the emergence of sustained oscillations after passing a certain critical parameter value. The model has proven to be dynamically very rich. For instance, when considering the space-time dynamics of the Brusselator as a reaction-diffusion equation the system exhibits successive bifurcations which induce \textit{dissipative structures} \cite{NicolisPrigogine}. Furthermore, under the influence of a periodic forcing, it may exhibit chaotic behaviour \cite{Young13}. Here, we explore a complementary route by assuming that one of the system's parameters is perturbed by Brownian noise, yielding a stochastic differential equation (SDE) as considered in \cite{ArnoldBrus}. 

We employ the perspective of \emph{random dynamical systems} (RDS) in order to understand the trajectory-wise behaviour of the solutions of this SDE on the state space $ X = \mathbb R^2$. 
Taking the time set $\mathbb{R}$, the RDS $(\theta,\varphi)$ consists of the noise model $(\Omega,\mathcal{B},\mathbb{P},\theta)$, where $(\Omega, \mathcal{B})$ is the canonical path space with Wiener measure $\mathbb{P}$, and a $\mathbb{P}$-invariant dynamical system $\theta=(\theta_t)_{t\in\mathbb{R}}$, and the family of maps
$\varphi^t_{\omega}:X\rightarrow X$ which determines the state of the system at time $t\geq 0$ for a fixed noise realization $\omega\in\Omega$.
The RDS approach allows to compare two trajectories with the same noise realizations $\omega \in \Omega$. In other words, given that the map $\varphi^t_{\omega}$ determines the evolution of the system, we want to compare $\varphi^t_\omega(x_0,y_0)$ and $\varphi^t_\omega(x_1,y_1)$ for distinct elements $(x_0,y_0)$ and $(x_1,y_1)$ of the state space $X$. Specifically, the RDS approach analyses whether these two trajectories synchronize or diverge asymptotically, determining the form of a \emph{random attractor} \cite{CrauelKloeden15}. We refer the reader to \cite{arnold1998} for a comprehensive exposition of RDS theory.

While it has been suggested numerically in \cite{ArnoldBrus} that the Brusselator with parametric noise exhibits noise-induced synchronization as $t\rightarrow\infty$, we adopt a different point of view by quantifying finite-time instabilities via the FTLEs. We observe numerically that the Hopf bifurcation of the stochastic Brusselator persists in a similar manner as in the cases previously mentioned -- the system exhibits instabilities in finite-time windows.

We perform our numerical explorations in two complementary directions: we calculate the FTLE for a fixed noise realization $\omega\in\Omega$, first by fixing a time length $T>0$ and varying the intial condition, and secondly by fixing an initial condition $(x,y)\in\mathbb{R}^2_+$ and seeing the evolution of the FTLE through time windows of increasing length. In the former scenario, it is shown that the state space has regions of instability (i.e. where the FTLE is positive) which are not only localized around the deterministic unstable equilibrium. In the latter, we show that the FTLEs may attain positive values for arbitrarily large time windows whenever one of the parameters is large enough. It is precisely this last setting where the Brusselator can be formulated as a \textit{slow-fast system}. We sketch the different time scales involved in the deterministic setting, and how they extend to the stochastic picture.

\subsection*{Structure of the paper}
The rest of the paper is structured as follows. In Section~\ref{SEC:brusselatorIntro} we introduce the notion of a \textit{chemical reaction network}, and its different modeling approaches, depending on the volume scaling under consideration. In particular, we present the Brusselator as a chemical reaction network, and the model that we analyse throughout the rest of this paper. In Section~\ref{SEC:FTLE} we introduce the notion of finite-time Lyapunov exponents, as an indicator of instabilities in bounded time windows. In Section~\ref{SEC:slow-fas}, we perform a change of coordinates which turns the stochastic Brusselator into a slow-fast system in standard form. We further sketch a possible mechanism which generates the finite-time instabilities observed via the finite-time Lyapunov exponents. Last, in Section~\ref{SEC:conclusion} we provide some concluding remarks with an outlook for future work.

\section{The Brusselator: a chemical oscillator}
\label{SEC:brusselatorIntro}

Consider the general model of $L$ chemical species $S_1, \dots, S_L$ 
such that the state of the whole reaction system at time $t \geq0$ is given by $Z(t) = (Z_1, \dots, Z_L(t)) \in \mathbb{N}_0^L$ with $Z_l(t)$ 
denoting the number of molecules of type $S_l$ at time $t$. Reactions $R_k$, changing the number of molecules of a subset of species, occur at rates $\alpha_k(x)$ which are defined as propensity funtions involving reaction rate constants $\gamma_k$.
Assuming that these functions depend only on the current value of $Z(t)$, we obtain the \emph{time-change representation of the Markov jump process}
\begin{equation} \label{eq:tcr}
Z(t) = Z(0) + \sum_{k=1}^K \mathcal P_k \left( \int_0^t \alpha_k (Z(s)) \rmd s \right) \nu_k,
\end{equation}
where $\nu_k$ is the state change vector of reaction $R_k$ and $\mathcal P_k$, $k=1, \dots,K$, denote independent Poisson processes with unit rate.
For high accumulations of reaction events, the Poisson variables can be approximated by normal random variables, yielding the \emph{chemical Langevin equation} in It\^o form
\begin{equation} \label{eq:CLE}
\rmd Z(t) = \sum_k \alpha_k (Z(t)) \nu_k \rmd t + \frac{1}{\sqrt{V}}\sum_k \sqrt{\alpha_k (Z(t))} \nu_k \rmd W_k(t),
\end{equation}
where each $W_k(t)$ denotes a one-dimensional Wiener process and the state space for solutions of equation~\eqref{eq:CLE} is now extended to $\mathbb{R}^L$.
In the thermodynamic limit where volume and volume-dependent molecule numbers scale in a fixed proportion approaching infinity, we may average over the noise terms to obtain the purely deterministic version of equation~\eqref{eq:CLE}, i.e.~the ODE
\begin{equation} \label{eq:RRE}
\rmd Z(t) = \sum_k \alpha_k (Z(t))\cdot  \nu_k \rmd t,
\end{equation}
also called \emph{reaction rate equation}. For chemical reaction networks with different volume-dependent scalings, the three description levels~\eqref{eq:tcr}, \eqref{eq:CLE} and~\eqref{eq:RRE} can also be combined into hybrid models \cite{winkelmann2017hybrid}. 
For an introductory exposition of chemical reaction networks from these different perspectives see \cite{winkelmann2020}.

In this work we focus on the chemical reaction network
\begin{equation}\label{eq:BRUSSsystem}
\begin{array}{lrcl}
R_1: &	A & \rightarrow& X, \\
R_2: &	B+X & \rightarrow& Y+D,\\
R_3: &	2X+Y & \rightarrow& 3X,\\
R_4: &	X & \rightarrow& E.
\end{array}
\end{equation}
We refer to (\ref{eq:BRUSSsystem}) as the \textit{Brusselator reaction network}. As classically considered, we suppose that $A$ and $B$ satisfy the \textit{infinite pool assumption}, for which an inexhaustible supply of both reactants is injected into the system so that their number of molecules (or their concentration when properly rescaled with the volume of the system) are assumed to be constant parameters. Under this assumption, system~\eqref{eq:BRUSSsystem} becomes an \textit{open system}. It is relevant to note here that both $E$ and $D$ are products of the system which do not interact back with the rest of the reactants. For this reason, we focus only on the variables $X$ and $Y$ in the rest of the paper.

The Brusselator was originally proposed in \cite{Prigogine68} as an example of a chemical system involving only two species which under suitable conditions is driven out of the thermodynamic equilibrium. The exploration of such type of reaction networks becomes relevant since many biochemical systems depend essentially on two main species. For such systems which involve only uni- and bimolecular reactions, it is known that if the rate functions $\alpha_k$ obey the \textit{law of mass action} the dynamics given by their reaction rate equations (cf. \eqref{eq:RRE}) does not admit attracting limit cycles, see \cite[Chapter 7]{NicolisPrigogine}. The intermediate reaction channel $R_3$ thus provides a simple trimolecular step for which interesting dynamics may emerge. Notice, however, that the reaction $R_3$ is physically unrealistic, making the Brusselator only a prototype of chemical instabilities.

\subsection{The reaction jump process}
Following \cite{winkelmann2020}, for each of the reaction channels $R_k$ ($k=1,2,3,4$) we define the vectors $\vec{v}_k$ as
\[
\begin{array}{cc}
	\begin{array}{rcl}
        \nu_1&=&(1,0,0,0),\\
	\nu_2&=&(-1,1,1,0),
        \end{array} 
        &
	\begin{array}{rcl}
        \nu_3&=&(1,-1,0,0),\\
	\nu_4&=&(-1,0,0,1).\\
        \end{array}
\end{array}
\]
According to the law of mass action, the corresponding reaction rates $\alpha_k$ for each reaction channel $R_k$ are given as \cite{Dey11}
\[
\begin{array}{cc}
    \begin{array}{rcl}
	\alpha_1(Z) &=& \gamma_1 A\\
	\alpha_2(Z) &=& \gamma_2V^{-1}B\cdot X
    \end{array} &
    \begin{array}{rcl}
	\alpha_3(Z) &=& \gamma_3V^{-2}\cdot X(X-1)Y \\
	\alpha_4(Z) &=& \gamma_4 \cdot X,\\
    \end{array}
\end{array}	
\]
where $Z=(X,Y,D,E)$, and, abusing the notation, the variables $Z_i(t)$ denote the number of molecules of species $Z_i\in\lbrace X,Y,D,E \rbrace$ at time $t$. The constants $\gamma_i$ for $i=1,2,3,4$ are positive real numbers, $A,B\in\mathbb{N}$, and $V>0$ denotes the volume of the system. The evolution of the system is thus given by the Markov jump process \eqref{eq:tcr}. 

The Brusselator reaction network, modeled as a Markov jump process, has been numerically studied in \cite{gillespie1977exact}, whose oscillatory behaviour was compared with the \textit{Lotka-Volterra system}. The numerical scheme employed is called \textit{the direct method}, and splits the jump process into two stochastic processes $(Z_n), (T_n)$, where $T_n$ is the time at which a reaction $R_k$ occurs and $Z_n=Z(T_n)$ is the state of the system which remains constant until another reaction occurs. Both quantities evolve according to a random difference equation, and thus define an RDS, see \cite{EngelOlicon22} for general results in this direction.

\subsection{Reaction rate equations}
By considering the thermodynamic limit when $V\rightarrow\infty$, the reaction rate equation \eqref{eq:RRE} for the Brusselator reaction network reads as
\begin{equation}
\label{eq:BrusDet}
\begin{array}{rcl}
    \dot{X}&=&a-(1+b)X+X^2Y, \\
    \dot{Y}&=&bX-X^2Y.
\end{array}
\end{equation}
System \eqref{eq:BrusDet} has been thoroughly studied \cite{NicolisPrigogine}, exhibiting the following key properties: it has a unique equilibrium point $(a,b/a)$ which is stable whenever $b<b_{crit}:=1+a^2$ and unstable when $b>b_{crit}$. An attracting limit cycle exists precisely in the latter case due to a \textit{Hopf bifurcation}, see Figure~\ref{fig:detBrusselator}.

\begin{figure}[ht]
     \centering

\begin{subfigure}[b]{0.48\textwidth}
        \centering
      \begin{overpic}[width=\textwidth]{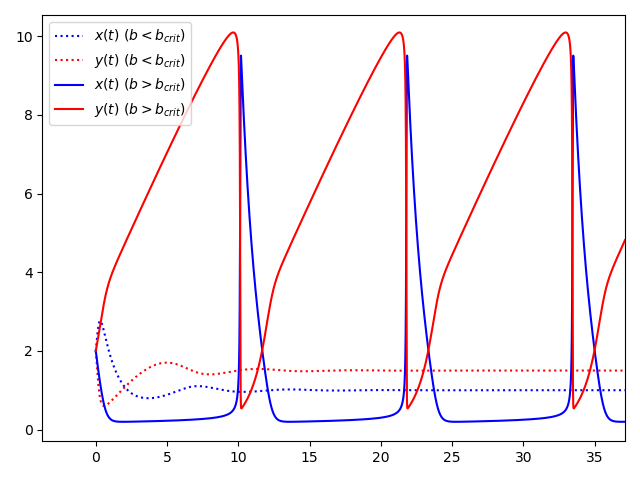}
        \put(52,-3){\scriptsize $t$}
        \end{overpic}
        \caption{}
    \end{subfigure}
     \hfill
     \begin{subfigure}[b]{0.47\textwidth}
         \centering
          \begin{overpic}[width=\textwidth]{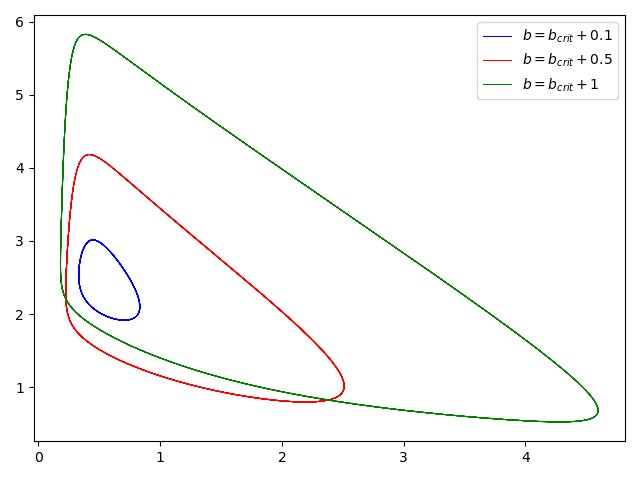}
        \put(-2,40){\scriptsize $y$ }
        \put(50,-2){\scriptsize $x$}
        \end{overpic}
        \caption{}
    \end{subfigure}
     \hfill
       \caption[Solutions of the Brusselator]{In (A), the time series of two different orbits of the system for two different parameter values respectively. Oscillatory behaviour arises when $b>b_{crit}$, and exhibits a clear time scale separation when $b\gg b_{crit}$. In (B), different profiles of the attracting limit cycle are depicted depending on the value of $b$. The limit cycles grow as $b$ becomes larger.}
        \label{fig:detBrusselator}
\end{figure}

\subsection{The chemical Langevin equation}
Consider now the rescaled variable $(X,Y)\mapsto \frac{1}{V}(X,Y)$. For large values of $V$ the Markov jump process can be approximated by the chemical Langevin equation \eqref{eq:CLE}, which for the Brusselator reads as the It\^o SDE
\begin{equation}\label{eq:BRUSSLangevin}
	\begin{array}{lll}
		dX_t&=&\left(\tilde{\gamma}_1\tilde{A}-\tilde{\gamma}_2\tilde{B}X_t+\tilde{\gamma}_3X_t(X_t-\frac{1}{V})Y_t-\tilde{\gamma}_4X_t\right)dt +\sqrt{\frac{\tilde{\gamma}_1\tilde{A}}{V}}d\tilde{W}_t^1\\
		& & -\sqrt{\frac{\tilde{\gamma}_2\tilde{B}}{V}X_t}d\tilde{W}_t^2+\sqrt{\frac{\tilde{\gamma}_3}{V}X_t(X_t-\frac{1}{V})Y_t}d\tilde{W}_t^3-\sqrt{\frac{\tilde{\gamma}_4}{V} X_t}d\tilde{W}_t^4,\\
  \\
		dY_t &=& \left(\tilde{\gamma}_2\tilde{B}X_t-\tilde{\gamma}_3X_t(X_t-\frac{1}{V})Y_t\right)dt +\sqrt{\frac{\tilde{\gamma}_2\tilde{B}}{V}X_t}d\tilde{W}_t^2\\
		 & &-\sqrt{\frac{\tilde{\gamma}_3}{V}X_t(X_t-\frac{1}{V})Y_t}d\tilde{W}_t^3,
	\end{array}
\end{equation}
where the rescaled variables $\tilde{A}=\frac{1}{V}A$ and $\tilde{B}=\frac{1}{V}B$ are assumed to be constant parameters independent of the volume $V$, while for $i=1,2,3,4$ the quantities $\tilde{\gamma}_i$ are positive constants, and $\tilde{W}_t^i$ are mutually independent one-dimensional Wiener processes.

While solutions of the Markov jump process and the reaction rate equations are typically well defined for all $t\geq 0$, this is very often not the case for the chemical Langevin equation.  In particular, solutions of \eqref{eq:BRUSSLangevin} may not exist for all $t>0$, since the noisy perturbations push the trajectories through the boundary of the domain $\mathbb R_+^2$ with positive probability. This is clearly an inconsistency with the chemical interpretation of $X$ and $Y$ as concentrations.
There have been attempts to resolve this issue. 
For example, in \cite{wilkieWong08} a modified Langevin equation has been proposed mainly for uni- and bimolecular reactions which respect the positivity of solutions.
However, note that the Brusselator involves a trimolecular reaction, namely $R_3$. As a different strategy, a method with reflecting boundary conditions has been considered in \cite{LeiteWilliams2019}. This last proposal does not come from first principles, and some interpretation of the original problem may be lost.

\subsection{Brusselator with parametric noise}
In the rest of this work, we focus on a stochastic version of equation~\eqref{eq:BrusDet}, where the parameter $b$ is replaced by a perturbation with Brownian noise of Stratonovich type $b+\sigma \circ dW_t$. More precisely, in this work we study the system \cite{ArnoldBrus}
\begin{equation}
    \label{eq:ArnoldBruselator}
        \begin{array}{rcl}
            \rmd x_t &=& \left(a-(1+b)x_t+x_t^2y_t\right)\rmd t
            -\sigma x_t\circ \rmd W_t,\\
            \rmd y_t &=& \left(bx_t-x_t^2y_t\right)\rmd t
            +\sigma x_t\circ \rmd W_t.\\
            
        \end{array}
\end{equation}
We consider the stochastic Brusselator \eqref{eq:ArnoldBruselator} to be an intermediate model between the deterministic system \eqref{eq:BrusDet} and the chemical Langevin equation \eqref{eq:BRUSSLangevin}, guaranteeing the existence of  solutions that remain in the positive quadrant $\mathbb{R}_+^2$ for all positive times. The stochastic perturbations of this form still induce intriguing dynamical features that we will explore in the rest of this work.  

The solutions of \eqref{eq:ArnoldBruselator} induce a RDS as indicated in the Introduction. Formally speaking, we consider the noise space $\Omega=C_0(\mathbb{R},\mathbb{R}^2)$ endowed with is canonical $\sigma$-algebra $\mathcal{B}$, the Wiener measure $\mathbb{P}$, and the shift maps $\theta=(\theta_t)_{t\in\mathbb{R}}$ given by
\[
    \theta_t\omega(\cdot)=\omega(\cdot+t)-\omega(t),
\]
which are $\mathbb{P}$-invariant and ergodic. The state of the system is given by a measurable function $\varphi:\mathbb{R}_+\times\Omega\times \mathbb{R}^2\rightarrow \mathbb{R}^2$, called the \textit{cocycle map over $\theta$}, such that for all $t,s\in\mathbb{R}_+$ and $\omega\in\Omega$,
    \[
        \varphi^0_{\omega}=\id, \qquad \varphi^{t+s}_\omega=\varphi^t_{\theta_s\omega}\circ \varphi^s_{\omega},
    \]
where we adopt the notation $\varphi^t_\omega=\varphi(t,\omega,\cdot)$. The pair $(\theta,\varphi)$ is a \textit{smooth RDS} if the cocycle map $\varphi^t_{\omega}$ is $C^1$ for all $t\in\mathbb{R}_+$ and $\omega\in\Omega$.

We recall the following result.
\begin{theorem}
    \label{THM:RDS_Brusselator}
        The SDE \eqref{eq:ArnoldBruselator} generates a smooth RDS in $\mathbb{R}^2_+$. Moreover, system \eqref{eq:ArnoldBruselator} admits at most one stationary distribution.
\end{theorem}
\begin{proof}
    See \cite{ArnoldBrus}.
\end{proof}
Theorem~\ref{THM:RDS_Brusselator} is a consequence of standard existence results for SDEs. In fact, it follows that solutions exist locally backwards in time, so that for each $(\omega,x,y)\in\Omega\times\mathbb{R}_+^2$ there is $\tau_{-}(\omega,x,y)<0$ for which the RDS can be extended up to the interval $(\tau_-(\omega,x,y),\infty)$. Nevertheless, we emphasize that the shift maps $\theta_t$ on the base noise space $\Omega$ are defined for all $t\in\mathbb{R}$, and it is sufficient for the cocycle map to be defined in $\mathbb{R}_+$.

\begin{remark}
\label{REMARK:conjecture}
    In \cite{ArnoldBrus} it is conjectured that in fact, for any $a,b,\sigma>0$ there is a unique stationary distribution $\rho$ with corresponding invariant measure $\mu$ for the RDS, whose \textit{sample measures} $\mu_{\omega}$ are given by delta distributions $\mu_\omega=\delta_{A(\omega)}$. In particular, the random equilibrium point $A(\omega)$ would be such that for all $t\geq0$
    \[
        \varphi^t_\omega(A(\omega))=A(\theta^t\omega), \quad \mathbb{P}-a.s.
    \]
    It has been further conjectured that the top Lyapunov exponent is negative, implying that this random equilibrium is exponentially attracting \cite{Newman2020}, and thus every pair of trajectories synchronizes exponentially fast as $t\rightarrow\infty$. 
    While it has been shown in \cite{ArnoldBrus} that there are initial conditions for which solutions do not exist for all negative times with positive probability, numerical explorations suggest that, in fact, the random equilibrium $A(\omega)$ constitutes the only trajectory whose dynamics can be extended for all $t<0$.
\end{remark}

In order to visualise the trajectories of the system we employ the Euler-Maruyama method \cite{KloedenPlaten92} which approximates the solution of a general It\^o SDE
\begin{equation}
\label{eq:GeneralIto}
	dX_t=f(X_t,t)+\sum_{i=1}^m{g_i(X_t,t)\; \rmd W_t^i},
\end{equation}
with vector fields $g_i=(g_{1i},g_{2i},\ldots,g_{ni})^\top$ and a fixed initial condition $X_0$ at $t=0$, via a recursive random map
\begin{equation}
\label{eq:Euler_Maruyama}	X_{n+1}=X_n+hf(X_n,t_n)+\sum_{i=1}^m g_i(X_n,t_n)\; \Delta W_n^i,
\end{equation}
where $\left(\Delta W_n^i\right)_{n=0}^{\infty}$ are independently and identically distributed normal random variables with mean $0$ and variance $h$ for step size $h > 0$ sufficiently small, and $t_n=nh$.

Hence, for numerical solutions we convert \eqref{eq:ArnoldBruselator} to its It\^o version using the \textit{Wong-Zakai corrections}. More precisely, for an It\^o SDE in $\mathbb{R}^n$ of the form \eqref{eq:GeneralIto}
where $f,g_i:\mathbb{R}^n\times \mathbb{R}\rightarrow \mathbb{R}^n$ for $i=1,2,\ldots,m$ are  smooth, let $g:\mathbb{R}^n\times\mathbb{R}\rightarrow \mathbb{R}^{n\times m}$ be the matrix whose columns consist of the vector fields $g_i$ for $i=1,\ldots,m$, and let $W_t:=(W_t^1,\ldots,W_t^m)^\top$. The It\^o and Stratonovich integrals are component-wise related for $l=1,2,\ldots,n$ by 
\begin{equation}
\label{eq:WongZakai}
	\left(\int_0^t{g(X_s,s)\circ dW_s}\right)_l=\left(\int_0^t{g(X_s,s) dW_s}\right)_l+\frac{1}{2}\int_0^t\sum_{k=1}^m\sum_{j=1}^n \frac{\partial g_{lk}}{\partial X_j}(X_t,t)\cdot g_{jk}(X_t,t)
\end{equation}

By using \eqref{eq:WongZakai} for \eqref{eq:ArnoldBruselator} we obtain the It\^o SDE
\begin{equation}
\label{eq: ArnoldBrusIto}
	\begin{array}{rcl}
		dx&=&[a-(1+b)x+x^2y+\frac{\sigma^2}{2}x]dt-\sigma x dW_t, \\
		dy&=& [bx-x^2y-\frac{\sigma^2}{2}x]dt+\sigma x dW_t,
\end{array}
\end{equation}
which we can use for numerical simulations of the induced random dynamical system. We start by exploring the \emph{two-point motion}, i.e.~comparing two trajectories $\varphi_{\omega}^t(x_0,y_0)$ and $\varphi_{\omega}^t(x_1,y_1)$ for a fixed noise path $\omega \in \Omega$. Figure~\ref{fig:BRUSSorbits} depicts such two trajectories for parameter values before and after the Hopf bifurcation.

\begin{figure}[ht]
     \centering
     \begin{subfigure}[b]{0.5\textwidth}
         \centering
             \begin{overpic}[width=\textwidth]{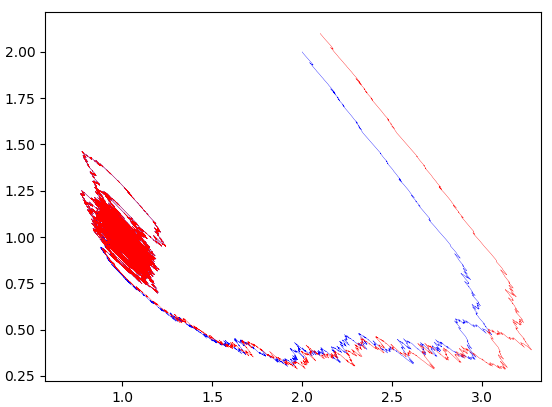}
        \put(52,-2){\scriptsize $x$}
        \put(-2,40){\scriptsize $y$}
        \end{overpic}
        \caption{}
         \label{subfig:equilibrium_stoch}
     \end{subfigure}
     \hfill
     \begin{subfigure}[b]{0.47\textwidth}
         \centering
         \begin{overpic}[width=\textwidth]{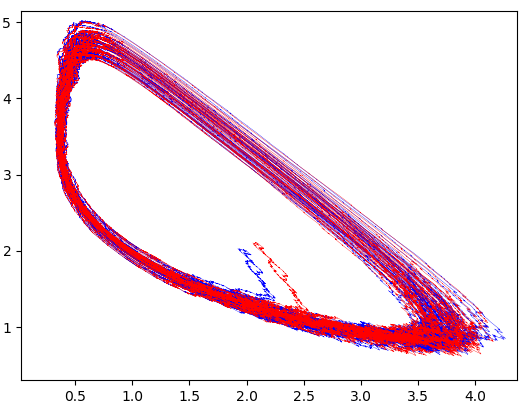}
        \put(52,-2){\scriptsize $x$}
        \put(-2,40){\scriptsize $y$}
        \end{overpic}
        \caption{}
         \label{subfig:oscillation_stoch}
     \end{subfigure}
     \hfill
     
       \caption[Two orbits for the stochastic Brusselator system]{
       Here we portray two different orbits of \eqref{eq: ArnoldBrusIto} with neighbouring initial conditions, which are subject to the same noise realization. We take a step size $h=10^{-3}$ and iterate the Euler-Maruyama scheme from $t_0=0$ to $T=150$. We take $\sigma=0.1$ and $a=1$. In (A), we consider the parameter $b=b_{crit}-1$ and the orbits tend to stay around the deterministic equilibrium point. In (B), we have $b=b_{crit}+1$, and the orbits exhibit fluctuations around the deterministic limit cycle. }
        \label{fig:BRUSSorbits}
\end{figure}

The qualitative picture in Figure~\ref{fig:BRUSSorbits} for $b>b_{crit}$ has some interesting features. On the one hand, given an initial condition $(x,y)\in\mathbb{R}_+^2$ the trajectory $\varphi^t_{\omega}(x,y)$ clearly resembles the triangle-like shape of the deterministic limit cycle, cf. Figure~\ref{fig:detBrusselator} -- starting from the bottom right, moves to the left in a slow time scale until it takes a fast transition towards a more vertical-like movement. Once the trajectory is closer to the vertical axis, it heads towards its uppermost part in an ultraslow fashion. Once on its apex, the system ``slides down'' almost in a straight diagonal line back to the bottom right in a fast time scale, where the trajectory repeats the process. However, due to the random perturbation of the system, the trajectory seems to choose different highest points, and times to reach them on each repetition. A similar behaviour for the Brusselator has been observed previously in \cite{gillespie1977exact} at the level of the Markov jump process. A more detailed description of the dynamics seen in Figure~\ref{fig:BRUSSorbits} in terms of a fast-slow system will be given in Section~\ref{SEC:slow-fas}, see also Figure~\ref{fig:Time_scales}.

In order to explore in more detail dynamical features in terms of noise-induced instabilities or synchronization, we calculate the distance between trajectories 
\[d_t\left( (x_0,y_0),(x_1,y_1)\right) = \|\varphi_{\omega}^t(x_0,y_0) -\varphi_{\omega}^t(x_1,y_1)\|,
\]
as shown in Figure~\ref{fig:BRUSSdistance}. We observe that, in both cases before and after the bifurcation, any two arbitrary orbits approach each other eventually as $t\rightarrow \infty$. 
However, the fashion in which this is done is qualitatively different. Indeed, when $b<b_{crit}$ the convergence between orbits is uniform, while in the case when $b>b_{crit}$ the orbits deviate repeatedly from each other, but they do so with less strength as time grows. In spite of that, the amplitude of their distance may become larger again at intermediate time windows.

\begin{figure}[ht]
     \centering
     \begin{subfigure}[b]{0.47\textwidth}
         \centering
         \begin{overpic}[width=\textwidth]{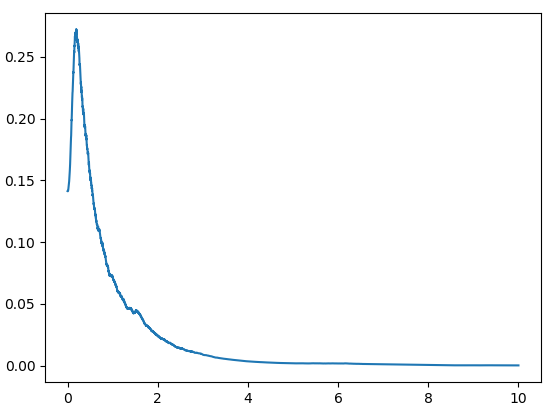}
        \put(52,-2){\scriptsize $t$}
        \put(-4,35){\scriptsize $d_t$}
        \end{overpic}
        \caption{}
         \label{subfig:distance eq}
     \end{subfigure}
     \hfill
     \begin{subfigure}[b]{0.45\textwidth}
         \centering
        \begin{overpic}[width=\textwidth]{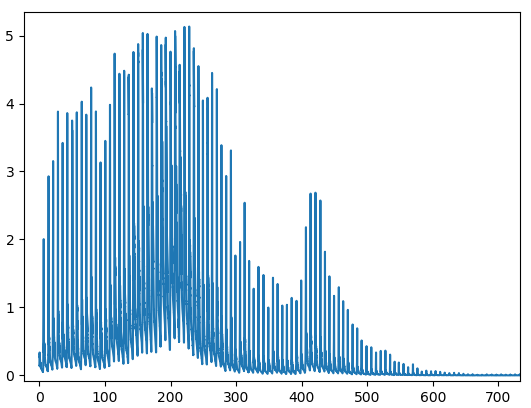}
        \put(52,-2){\scriptsize $t$}
        \put(-4,35){\scriptsize $d_t$}
        \end{overpic}
        \caption{}
         \label{subfig:distance cycle}
     \end{subfigure}
     \hfill
       \caption[Distance between orbits along time]{For the cases above, we depict here the distance between the two different orbits as a function of time. We consider $b=b_{crit}-1$ and $b=b_{crit}+1$ in (A) and (B), respectively. The rest of the parameters were chosen to be the same as in Figure \ref{fig:BRUSSorbits}. 
       Note the big difference between the scales in both axes.}
        \label{fig:BRUSSdistance}
\end{figure}

\section{Finite-time Lyapunov exponents}
\label{SEC:FTLE}
As mentioned in Remark~\ref{REMARK:conjecture} and sketched in Figure~\ref{fig:BRUSSdistance}, it was conjectured that any two trajectories of the system, under the influence of the same noise realization, synchronize exponentially fast for any parameter combination of $a, b$ and $\sigma$. However, Figure~\ref{fig:BRUSSdistance} still shows a clear qualitative difference between the case $b<b_{crit}$ and $b>b_{crit}$. Specifically, when $b>b_{crit}$ there seems to be a large time window where the two trajectories  deviate from each other recurrently before their distance finally starts approaching zero.

A useful quantity to study local finite-time stability in RDS are the finite-time Lyapunov exponents (FTLEs). For their computation, we consider the linearization along an orbit $\varphi^t_{\omega}(x,y)$ of (\ref{eq:ArnoldBruselator}). It is straightforward to see that $$v_t= \D \varphi^t_{\omega}(x,y)\cdot v_0, \quad v_0\in\mathbb{R}^2,$$  
satisfies the \textit{variational equation} for each $(x,y)\in \mathbb{R}_+^2$ and $\mathbb{P}$-a.e $\omega \in \Omega$ which, for system~\eqref{eq:ArnoldBruselator}, reads
\begin{equation}
\label{eq:variational}
	\rmd v=J\left( \varphi_{\omega}^t(x,y) \right)v \rmd t +\sigma \left[ \begin{array}{cc}
	-1 & 0\\
	1 & 0
\end{array}	 \right]v\circ \rmd W_t,
\end{equation}
where $J$ is the Jacobian matrix of the drift given by
\[
	J(x,y)=\left[ \begin{array}{cc}
		-(1+b)+2xy & x^2\\
		b-2xy & -x^2
\end{array}	 \right].
\]
We can also consider the matrix form of \eqref{eq:variational} as the SDE
\begin{equation}
    \label{eq: variational_matricial}
    \rmd\Phi = J\left( \varphi^t_\omega(x,y)\right)\Phi \rmd t + \sigma \left[ \begin{array}{cc}
	-1 & 0\\
	1 & 0
\end{array}	 \right]\Phi\circ \rmd W_t,
\end{equation}
where $\Phi\in\mathbb{R}^{2\times 2}$. 
\begin{definition}
    \label{DEF:FTLE}
        Let $\varphi^t_\omega(x,y)$ be a trajectory of \eqref{eq:ArnoldBruselator}, and let
        $\Phi_t$ be the solution of \eqref{eq: variational_matricial} at time $t$, with initial condition $\Phi_0=\id$. The (maximal) finite-time Lyapunov exponent (FTLE) of \eqref{eq:ArnoldBruselator} for $(\omega,x,y)$ at time $T>0$ is given by 
        \begin{equation}
            \label{eq:FTLE}
                \lambda^T(\omega,x,y):=\frac{1}{T}\ln\left\Vert \Phi_T\right\Vert,
        \end{equation}
    where $\Vert\cdot\Vert$ is the operator norm or, equivalently, its largest singular value.
\end{definition}

Note that the classical \textit{maximal Lyapunov exponent} is given by
\[
	\lambda(\omega,x,y)=\limsup_{T\rightarrow\infty} \lambda^T(\omega,x,y).
\]
Again, in order to solve numerically \eqref{eq: variational_matricial} above one must convert the system into its It\^o version using (\ref{eq:WongZakai}). We thus get the system
\[
	\begin{array}{rcl}
		\rmd x&=&[a-(1+b)x+x^2y+\frac{\sigma^2}{2}x]\rmd t-\sigma x \rmd W_t \\
		\rmd y&=& [bx-x^2y-\frac{\sigma^2}{2}x]\rmd t+\sigma x \rmd W_t.
	\\
	\rmd v&=&\left\lbrace J\left( \varphi_{\omega}^t(x_0,y_0)\right)+\frac{\sigma^2}{2}\left[  \begin{array}{cc}
		1 & 0\\
		-1 & 0
	\end{array} \right] \right\rbrace v\rmd t + +\sigma \left[ \begin{array}{cc}
	-1 & 0\\
	1 & 0
\end{array}	\right] v\cdot \rmd W_t.
\end{array}
\]
By fixing a noise realization, we compute numerically the FTLEs throughout the state space by implementing the Euler-Maruyama method, see \eqref{eq:Euler_Maruyama}. In Figure~\ref{fig:FTLElandscape}, the FTLE \eqref{eq:FTLE} is numerically calculated over a mesh of fine grid of initial conditions in the state space. We use the period of the deterministic limit cycle $\tau$ as a reference, so that $T\approx\tau/2$. For $ b >b_{crit}$ (Figure~\ref{fig:FTLElandscape}(A)), we observe an area of instability around the limit cycle within a certain region whose role is explored in more detail later, when exploring the multiscale dynamics. For $ b <b_{crit}$ (Figure~\ref{fig:FTLElandscape}(B)), no such instabilities can be found.
Note that the change of behaviour in terms of the FTLEs is really measured around the critical attracting object, i.e.~the equilibrium before the bifurcation and the limit cycle after the bifurcation. This is in contrast to where the change of stability is only measured around the equilibrium (see e.g.~\cite{Doanetal}). 

In Figure~\ref{fig:FTLESeries} we explore the variation of the FTLE with respect to $T$, by fixing both $\omega\in\Omega$ and $(x,y)\in\mathbb{R}^2_+$. Here, an interesting phenomenon is revealed: as the parameter $b$ increases, it seems to allow the FTLEs to attain positive values at much larger values of $T$. In other words, as $b$ grows the system exhibits finite-time instabilities in larger time windows. It remains an open question whether this can be achieved in the limit as $T\rightarrow\infty$ for an adequate rescaling of both $b$ and $\sigma$, showing that for some values of the parameters the classical Lyapunov exponent is positive. In the following, we explore the case in which $b\rightarrow\infty$.

\begin{figure}[ht]
     \centering
     \begin{subfigure}[b]{0.45\textwidth}
         \centering

          \begin{overpic}[width=\textwidth]{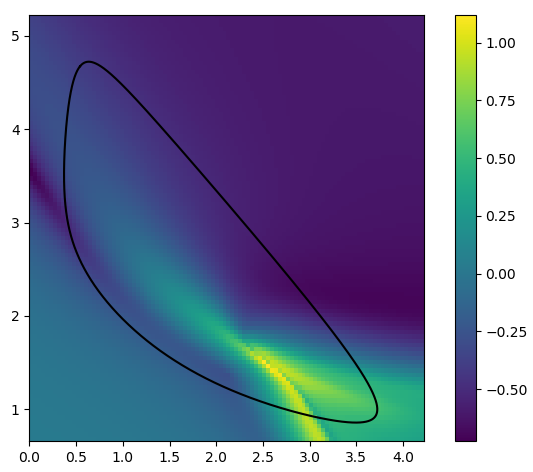}
        \put(50,-2){\scriptsize $x$}
        \put(-2,40){\scriptsize $y$}
        \end{overpic}
        \caption{}
         \label{subfig:FTLE_Thalf}
     \end{subfigure}
    \begin{subfigure}[b]{0.47\textwidth}
         \centering

          \begin{overpic}[width=\textwidth]{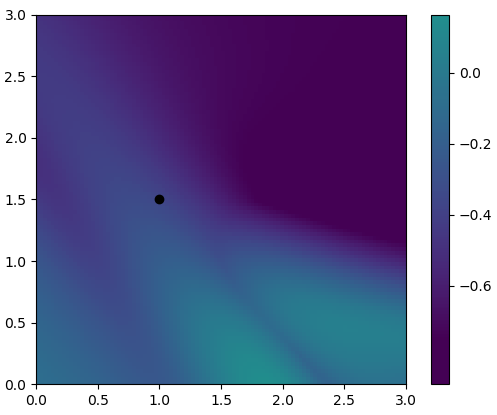}
        \put(50,-2){\scriptsize $x$}
        \put(-2,40){\scriptsize $y$}
        \end{overpic}
        \caption{}
         \label{subfig:FTLE_Eq}
     \end{subfigure}
 \caption[FTLE throughout the state space]{We depict a landscape of the maximal FTLE as given in Definition~\ref{DEF:FTLE}. Here we took $a=1$ and $\sigma=0.1$. The time window length considered was $T\approx \tau/2\equiv \frac{1}{2\Omega}$, where $\Omega$ is the dominant frequency obtained when applying the fast Fourier transform to a time series of the $x$ coordinate. We compute the FTLE throughout the state space in a neighbourhood of the deterministic invariant sets over a mesh of $100\times100$ points.  In (A), we took $b=b_{crit}+2$. For comparison, we portray in (B) the FTLE landscape for $b=b_{crit}-1$, for which the deterministic dynamics admits a globally attracting equilibrium.
}
        \label{fig:FTLElandscape}
\end{figure}

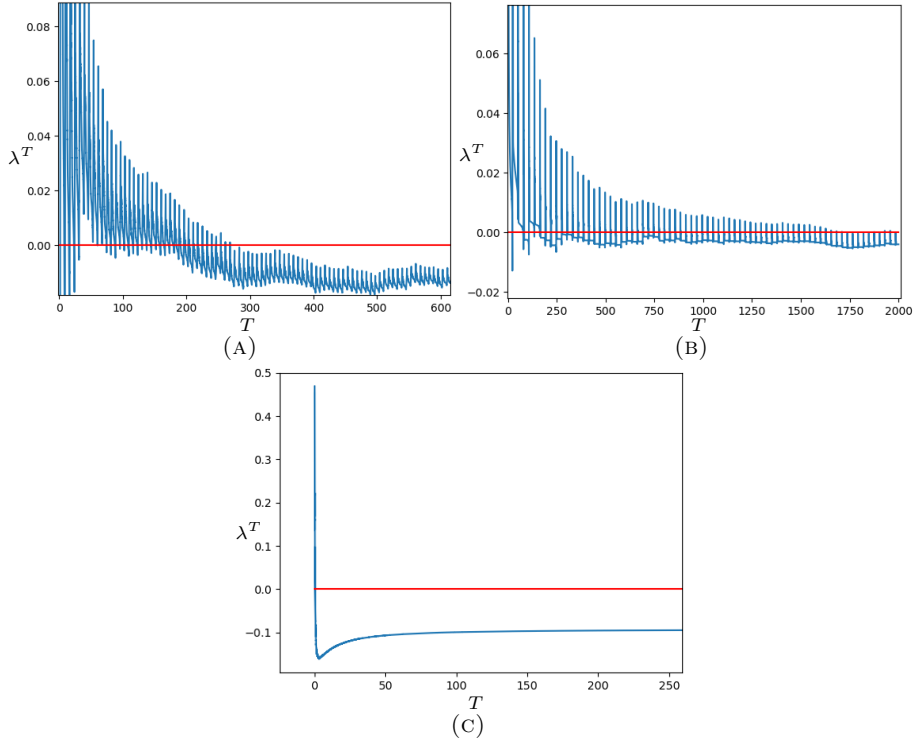
\begin{figure}[ht]
     \centering
     \begin{subfigure}[b]{0.46\textwidth}
         \centering
      \begin{overpic}[width=\textwidth]{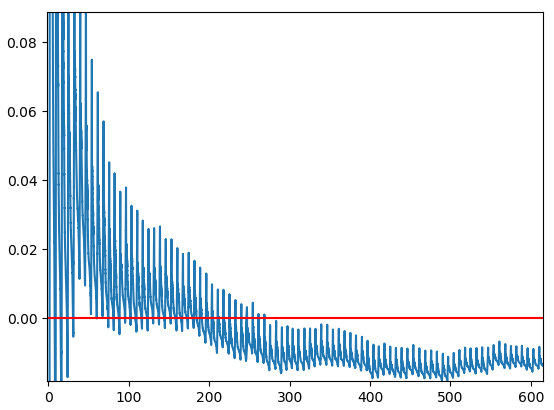}
        \put(50,-3){\scriptsize $T$}
        \put(-3,35){\scriptsize $\lambda^T$}
        \end{overpic}
        \caption{}
         \label{subfig:FTLE_timeseries1}
     \end{subfigure}
    \begin{subfigure}[b]{0.47\textwidth}
        \centering
         \begin{overpic}[width=\textwidth]{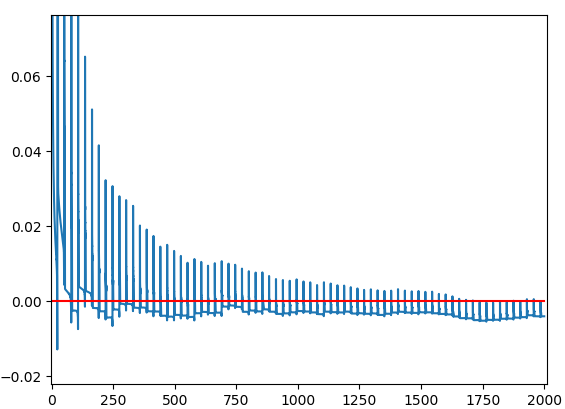}
        \put(50,-3){\scriptsize $T$}
        \put(-2,35){\scriptsize $\lambda^T$}
        \end{overpic}
        \caption{}
         \label{subfig:FTLE_timeseries2}
    \end{subfigure}

    \begin{subfigure}[b]{0.47\textwidth}
        \centering
         \begin{overpic}[width=\textwidth]{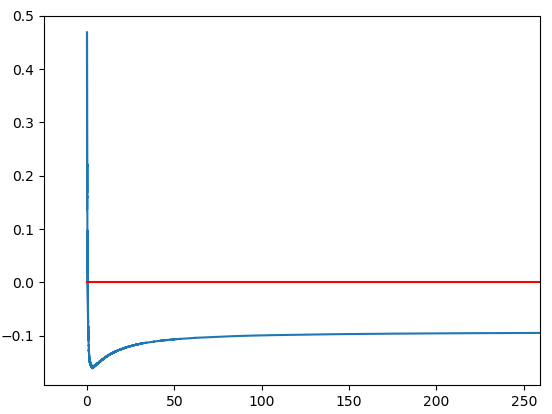}
        \put(50,-3){\scriptsize $T$}
        \put(-2,35){\scriptsize $\lambda^T$}
        \end{overpic}
        \caption{}
         \label{subfig:FTLE_timeseries3}
    \end{subfigure}
 \caption[FTLE for different values of $b$]{Here the FTLE is plotted as a function of $T$ for a fixed initial condition, considering $\sigma=0.1$, $a=1$, and a time step $h=10^{-3}$. In (A) we took $b=b_{crit}+1$, while in (B) we considered $b=b_{crit}+7$. The time windows in which the FTLE explores the positive real line is considerably larger in the latter case. As a comparison, in (C) we took $b=b_{crit}-1$, for which the system admits a globally attracting equilibrium.}
        \label{fig:FTLESeries}
\end{figure}

\section{A slow-fast structure for the Brusselator}
\label{SEC:slow-fas}
Consider again the deterministic system \eqref{eq:ArnoldBruselator}, with a change of coordinates (see also \cite{LiHouShen16}) given by
\begin{equation}
\label{eq:change_of_coordinates}
\left(\begin{array}{c}
    x \\
    y
\end{array}\right) \mapsto
\left(\begin{array}{c}
    y \\
    x+y
\end{array}\right)
\end{equation}
We thus obtain the SDE
\begin{equation}
\label{eq:SDE_change_of_variables}
    \begin{array}{rcl}
        \rmd x&=& \left[b(y-x)-x(y-x)^2\right]\rmd t+\sigma(y-x)\circ \rmd W_t,\\
        \rmd y&=&\left[a-(y-x)\right]\rmd t,
    \end{array}
\end{equation}
whose state space is now given by the set $\{(x,y) : x\geq0 \ y\geq x\}$. 
Based on the observations made in the previous section, we are interested in the case $b\gg b_{crit}$. 
Hence, we write $b=\frac{a}{\epsilon}$ with $\epsilon \ll 1$, so that system~\eqref{eq:SDE_change_of_variables} reads as
\begin{equation}
    \label{eq:slowDynamics}
    \begin{array}{rcl}
        \epsilon \rmd x&=&\left[a(y-x)-\epsilon x(y-x)^2\right] \rmd t+\epsilon \sigma(y-x)\circ \rmd W_t\\
        \rmd y&=& \left[a-(y-x)\right] \rmd t.
    \end{array}
\end{equation}
Observe that \eqref{eq:slowDynamics} has the structure of a singularly perturbed system in standard form, representing the dynamics on a slow time scale. By considering the fast time scale $\tau=t/\epsilon$, and recalling that $\Tilde{W}_\tau=\epsilon^{-1/2}W_{\epsilon \tau}$ is again a Wiener process for all $\epsilon$, we have the equivalent system (in the distributional sense) 
\begin{equation}
    \label{eq:fast dynamics}
    \begin{array}{rcl}
         \rmd x&=&\left[a(y-x)-\epsilon x(y-x)^2\right]\rmd\tau+\sqrt{\epsilon} \sigma(y-x)\circ \rmd\Tilde{W}_\tau\\
        \rmd y&=& \epsilon\left[a-(y-x)\right]\rmd\tau.
    \end{array}
\end{equation}

By taking $\epsilon=0$ in \eqref{eq:slowDynamics}, we obtain the \textit{reduced problem}
\begin{equation}
    \label{eq:reduced}
    \begin{array}{rcl}
     0&=&a(y-x),\\
     \rmd y &=& \left[a-(y-x)\right]\rmd t,
    \end{array}
\end{equation}
which is, in fact, a deterministic algebraic-differential equation. This yields the \emph{critical manifold}
\[
    \mathcal{S}_0=\{(x,y) \;\vert\; x=y\},
\]
see also Figure~\ref{fig:FastSLowFigure}. 

\subsection{Deterministic dynamics for $\epsilon\ll1$} 
As a first approximation, we sketch the dynamics in the deterministic system by considering $\sigma=0$ in \eqref{eq:slowDynamics}--\eqref{eq:fast dynamics}, that is we consider the deterministic slow-fast system in standard form
\begin{equation}
    \label{eq:slowfast_deterministic}
        \begin{array}{cc}
            \begin{array}{rcl}
            \epsilon\dot{x}&=&a(y-x)-\epsilon x(y-x)^2  \\
            \dot{y} &=& a-(y-x)
            \end{array} \quad &  
            \begin{array}{rcl}
            x' &=& a(y-x)-\epsilon x(y-x)^2\\
            y' &=& \epsilon[a-(y-x)],
             \end{array}
        \end{array}
\end{equation}
where $\dot{z}:=dz/dt$ and $z':=dz/d\tau$. In the next proposition we show that $\mathcal{S}_0$ is partially attracting.

\begin{prop}    \label{PROP:partial_hyperbolicity}
    The critical manifold $\mathcal{S}_0$ is normally hyperbolic. Its stable foliation has leaves
    \[
        W_s(z)=span\{(1,0)\},
    \]
    for each $z\in\mathcal{S}_0$. Furthermore, the dynamics of \eqref{eq:reduced} is given by
    \[
        x(t)=y(t)=x_0+at
    \]
\end{prop}
\begin{proof}
    Let $h(x,y)=(a(y-x),0)$. Its Jacobian matrix is calculated as
    \[
    Dh(x,y)=\left[ \begin{array}{cc}
        -a & a \\
         0 & 0
    \end{array} \right],
    \]
    for which the nontrivial eigenvalue is simply $\lambda=-a<0$, with eigenvector $(1,0)$. By considering $x=y$, then $\rmd y/\rmd t=a$ and the result follows.
\end{proof}

Fenichel's theorem (see for instance \cite[Theorem 2.4.2]{Kuehn15}) implies that for any compact submanifold $\hat{\mathcal{S}}_0$ of $\mathcal{S}_0$ there is $\epsilon_0$ sufficiently small such that for each $\epsilon\in(0,\epsilon_0)$ there is a normally hyperbolic locally invariant manifold $\hat{\mathcal{S}}_\epsilon$ which is an $\mathcal{O}(\epsilon)$-perturbation of $\mathcal{S}_0$. Since $\mathcal{S}_0$ is normally hyperbolic, we can take $\hat{\mathcal{S}}_0$ arbitrarily long. Yet, the flow makes a turn and traverses the nullcline $\dot{x}=0$, namely
\[
\mathcal{N}=\left\{(x,y) : y=x+\frac{a}{\epsilon x}\right\}.
\]
Since the nullcline has $\mathcal{S}_0$ as its asymptote as $x\rightarrow\infty$, whenever $\epsilon\ll 1$ the point in which the flow traverses the nullcline is still very close to $\mathcal{S}_0$, but its $x$-value is large. Hence, the dominating term for $x'$ is the cubic term $-\epsilon x(y-x)^2$ (see \eqref{eq:slowfast_deterministic}), and the flow heads to the left  until it reaches $\mathcal{N}$ again. After this, the flow follows $\mathcal{N}$ closely with a slow time scale, up to a turning point where it takes a fast transition back to a neighborhood of $\mathcal{S}_0$, and the cycle is reapeated. 

It is worth noticing that due to Proposition~\ref{PROP:partial_hyperbolicity}, no loss of hyperbolicity anticipates the turning from $\mathcal{S}_0$. 
The fact that the nullcline $\mathcal{N}$ approaches $\mathcal{S}_0$ as $x\rightarrow\infty$ suggests that the analysis should be performed near ``infinity'', by compactifying the state space. To the best of our knowledge, no rigorous singular perturbation analysis of the Brusselator has been done before, for which we refer to \cite{EngelOlicon23}, where by compactifying the state space, the ``line at infinity'' becomes a nonhyperbolic invariant manifold, which can be desingularized using an appropriate rescaling of the time variables. Similar arguments have been implemented in \cite{Kuehn14, Szmolyan09}.

\subsection{Stochastic Brusselator for $\epsilon\ll 1$.}
Let us sketch the stochastic dynamics of \eqref{eq:slowDynamics}--\eqref{eq:fast dynamics} for $\sigma>0$. First, notice that the noisy perturbation in \eqref{eq:slowDynamics} vanishes precisely on $\mathcal{S}_0$, which is an attracting normally hyperbolic invariant manifold for the deterministic system. It is intuitively clear that the trajectories of \eqref{eq:slowDynamics} starting close to $\mathcal{S}_0$, will remain so for some long time, whenever both $\epsilon$ and $\sigma$ are sufficiently small. Similarly to the deterministic case, our numerical explorations suggest that if a trajectory which is following closely the critical manifold $\mathcal{S}_0$ traverses the nullcline $\mathcal{N}$, it enters the fast regime and moves practically horizontally to the left until reaching $\mathcal{N}$ again. The process  then fluctuates around $\mathcal{N}$ until being sufficiently close to $\mathcal{S}_0$ again, see Figure~\ref{fig:FastSLowFigure}. 

Observe that this description matches completely with the one given in Section~\ref{SEC:brusselatorIntro} for Figure~\ref{fig:BRUSSorbits} by reverting the change of coordinates \eqref{eq:change_of_coordinates}. A matching between the corresponding time scales in systems \eqref{eq:ArnoldBruselator} and \eqref{eq:slowDynamics} is portrayed in Figure~\ref{fig:FastSLowFigure}.
\begin{figure}[ht]
     \centering
     \begin{subfigure}[b]{0.43\textwidth}
         \centering

        \begin{overpic}[width=\textwidth]{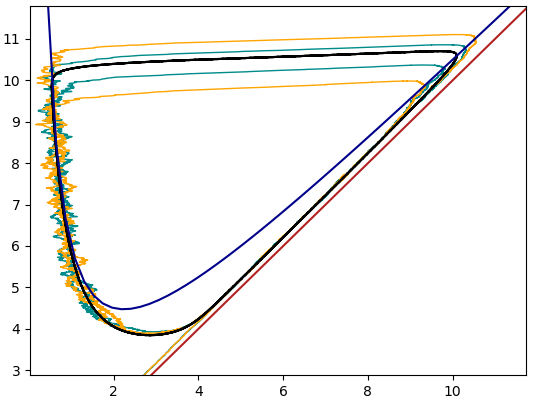}
        \put(50,-3){\scriptsize $x$}
        \put(-3,35){\scriptsize $y$}
        \end{overpic}
        \caption{}
         \label{subfig:Orbits_SlowFast} 
     \end{subfigure}    
    \begin{subfigure}[b]{0.45\textwidth}
        \centering
        \begin{overpic}[width=\textwidth]{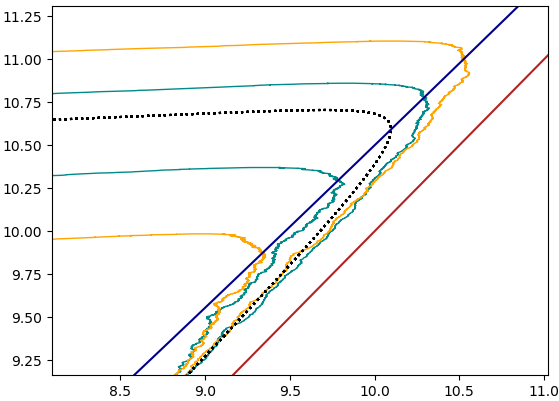}
        \put(50,-3){\scriptsize $x$}
        \put(-2,35){\scriptsize $y$}
        \end{overpic}
        \caption{}
         \label{subfig:Orbits_SlowFast_zoom} 
    \end{subfigure}
 \caption[Orbits in new coordinate system]{Two different orbits of system \eqref{eq:slowDynamics} with initial conditions $(2,2)$ and $(2.1,2.1)$ are obtained numerically and presented in (A) in blue and yellow, respectively. The critical manifold $\mathcal{S}_0$ is portrayed in red, the vertical nullcline in blue, and the deterministic limit cycle in black. A closer look at the \textit{turning points} is shown in (B). The parameters taken here are $a=1$, $b=b_{crit}+1$, and $\sigma=0.2$}
        \label{fig:FastSLowFigure}
\end{figure}

\subsubsection*{A mechanism for noise-induced (finite-time) instabilities} At last, let us sketch the mechanism by which the Brusselator exhibits finite-time instabilities. We refer the reader to Figure~\ref{fig:FastSLowFigure} for the time scale regimes of the state space.

Let us consider $(x_0,y_0)$ and $(x_1,y_1)$ two initial conditions for the system \eqref{eq:slowDynamics}, which are close to each other and also sufficiently close to $\mathcal{S}_0$. Let $\tau_{\mathcal{N}}(\omega,x_i,y_i)$ be the first hitting time of the trajectory $\varphi_{\omega}^t(x_i,y_i)$ ($i=0,1$) to the nullcline $\mathcal{N}$, so that for $t<\min\{\tau_{\mathcal{N}}(\omega,x_0,y_0), \tau_{\mathcal{N}}(\omega,x_1,y_1)\}$, it yields
\[d_t((x_0,y_0),(x_1,y_1))\approx  \| (x_0,y_0)-(x_1,y_1) \|^2. \]
In other words, the ultraslow regime (I) close to $\mathcal{S}_0$ is of a neutral nature, and no finite-time instabilities are observed in such regime. 

It is worth noticing that for $\epsilon$ sufficiently small, in Figure~\ref{fig:FastSLowFigure} the transitions from the ultraslow regime (I) to the subsequent fast regime (II) may occur at different altitudes and at different times, generating an almost instantaneous deviation between the trajectories if $\epsilon\ll 1$. Indeed, the trajectories $\varphi_\omega^t(x_i,y_i)$ ($i=0,1$) may traverse the deterministic nullcline $\mathcal{N}$ in points which are sufficiently away from each other, and therefore at different times. When one of them, say $\varphi^t_{\omega}(x_0,y_0)$, enters the fast regime first, it moves quickly towards $\mathcal{N}$ while $\varphi^t_{\omega}(x_1,y_1)$ may remain still close to $\mathcal{S}_0$. 

The slow regime (III) and the fast regime (IV) are the only parts of the state space where the system might present noise-induced stabilization. Not only a rigorous analysis of the FTLEs on the regimes (I)-(IV) would give an indication of the finite-time instabilities that the system exhibits, but it may be of use in order to prove or disprove the conjecture that the classical Lyapunov exponent is negative.

\begin{figure}[ht]
     \centering
     \begin{subfigure}[b]{0.45\textwidth}
         \centering

        \begin{overpic}[width=\textwidth]{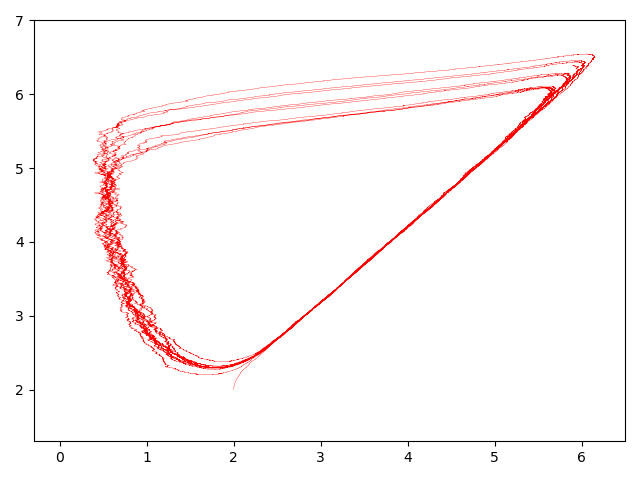}
        \put(50,-3){\scriptsize $x$}
        \put(-3,35){\scriptsize $y$}
        \put(35,65){\tiny (II) fast}
        \put(9,24){\tiny \begin{turn}{90} (III) slow \end{turn}}
        \put(20,11){\tiny (IV) fast}
        \put(50,21){\tiny \begin{turn}{45} (I) ultra slow \end{turn}}
        
        \end{overpic}
        \caption{}
         \label{subfig:timescales} 
     \end{subfigure}    
    \begin{subfigure}[b]{0.45\textwidth}
        \centering
        \begin{overpic}[width=\textwidth]{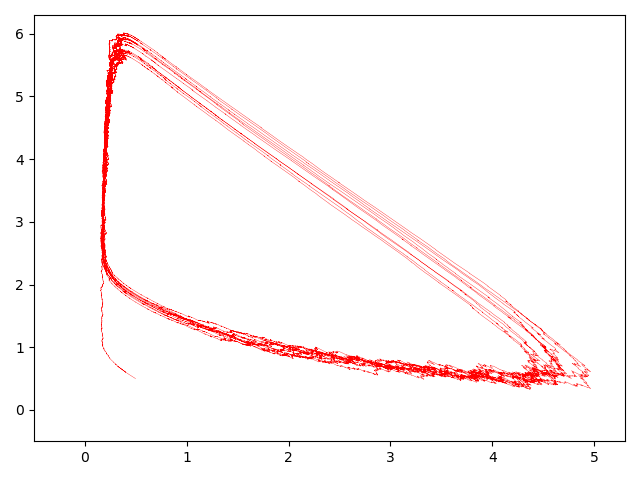}
        \put(50,-3){\scriptsize $x$}
        \put(-2,35){\scriptsize $y$}
        \put(9,34){\tiny \begin{turn}{90} (I) ultra slow\end{turn}}
        \put(46,56){\tiny \begin{turn}{-41}(II) fast\end{turn}}
        \put(52,10){\tiny (III) slow}
        \put(14,22){\tiny \begin{turn}{-18}(IV) fast\end{turn}}
        \end{overpic}
        \caption{}
         \label{subfig:timescales_original} 
    \end{subfigure}
 \caption[Orbits in new coordinate system]{The time scale separation in the stochastic Brusselator. In (A), an example of an orbit of the system \eqref{eq:slowDynamics} is portrayed, following the regimes (I), (II), (III), and (IV) in that order (counterclockwise). In (B), we present a prototypical trajectory in the original coordinates (clockwise), see \eqref{eq:ArnoldBruselator}. Each regime (I)-(IV) is identified via the change of coordinates \eqref{eq:change_of_coordinates}}
        \label{fig:Time_scales}
\end{figure}
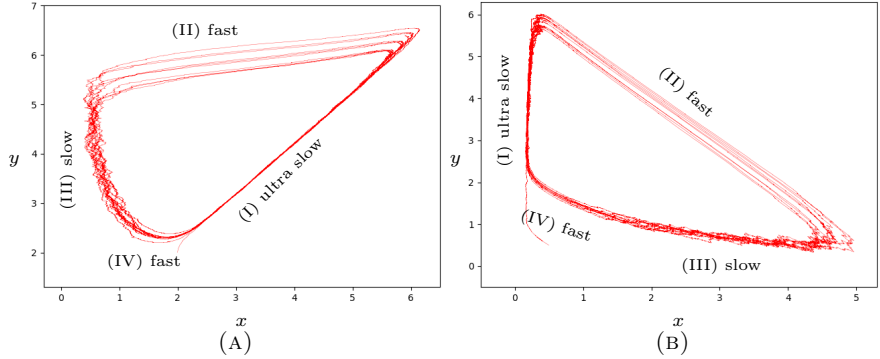

\section{Concluding remarks and outlook}
\label{SEC:conclusion}
In this work we have studied finite-time instabilities for a stochastic Brusselator via FTLEs. In \cite{ArnoldBrus} it was conjectured that \textit{parametric noise destroys the Hopf bifurcation}, in the sense that for any $\sigma>0$ the system \eqref{eq:ArnoldBruselator} admits a unique random equilibrium which is exponentially attracting. While this conjecture remains open, our numerical explorations show that finite-time instabilities occur in arbitrarily large time intervals as the parameter $b$ grows. 

In order to understand better the role of large $b$ already for the deterministic system, we have sketched a suitable singular perturbation analysis in Section~\ref{SEC:slow-fas}, introducing the parameter of time scale separation $\varepsilon \approx 1/b$. In \cite{EngelOlicon23}, under suitable rescalings of $x,y$ and $t$, we study the existence of a \textit{critical cycle} for the deterministic system from which the attracting periodic orbits are $\mathcal{O}(\epsilon)$-perturbations.

We consider it possible that an appropriate scaling of both the parameters $b$ and $\sigma$ may lead to \textit{noise-induced chaos} (cf.~\cite{ChemnitzEngel2022}), giving a negative answer to the conjecture posed in \cite{ArnoldBrus}. Along these lines, there exist extensions of singular perturbation analysis in the stochastic setting, see for instance \cite[Chapter 15]{Kuehn15}. This analysis for the stochastic Brusselator is left for future work.

While we have focused on the stochastic Brusselator \eqref{eq:ArnoldBruselator} defining a global RDS, it is still possible to consider the chemical Langevin equation \eqref{eq:BRUSSLangevin} and perform a similar exploration, despite solutions not being defined for all $t>0$. A suitable treatment of this situation may involve the notion of an \textit{absorbed Markov process} so that the process stops when one of the reactants becomes zero. In such a context, stationary distributions cannot exist and have to be replaced by the concept of \textit{quasi-stationary distributions}, for an introduction to the topic see \cite{Collet13}. On this basis, one can study changes of stability via \textit{conditioned Lyapunov exponents} \cite{EngelLambRasmussen2}.

At the level of the Markov jump process \eqref{eq:tcr}, the trajectories portrayed in \cite{gillespie1977exact} resemble those obtained in Figure~\ref{fig:BRUSSorbits}. As mentioned in \cite{gillespie1977exact}, these noisy oscillations have ``...the apparent inability of the system to consistently retrace its previous path on the diagonal leg of the limit cycle.'' In our setting, this property emerges as a consequence of the transition from a slow regime into a fast regime of the system, as exhibited in Figure~\ref{fig:FastSLowFigure}. To the best of our knowledge, this point of view has not been studied for the Markov jump process of the Brusselator reaction network, yielding another intriguing direction for future work.

\subsection*{Acknowledgements} We acknowledge the support of Deutsche Forschungsgemeinschaft (DFG) through CRC 1114 and under Germany's Excellence Strategy -- The Berlin Mathematics Research Center MATH+ (EXC-2046/1, project 390685689, in particular subproject AA1-8).  M.~E. additionally thanks the DFG-funded SPP 2298 for supporting his research. G.~O.-M. also thanks FU Berlin for a 3-month Forschungsstipendium. The authors gratefully acknowledge Peter Szmolyan and Bernd Krauskopf, for insightful discussions.

\bibliography{bibliography}

\providecommand{\bysame}{\leavevmode\hbox to3em{\hrulefill}\thinspace}
\providecommand{\MR}{\relax\ifhmode\unskip\space\fi MR }
\providecommand{\MRhref}[2]{%
  \href{http://www.ams.org/mathscinet-getitem?mr=#1}{#2}
}
\providecommand{\href}[2]{#2}
\begin{thebibliography}{10}

\bibitem{arnold1998}
L.~Arnold, \emph{Random dynamical systems}, Springer, Berlin, 1998.

\bibitem{ArnoldBrus}
L.~Arnold, G.~Bleckert, and K.~R. Schenk-Hopp\'e, \emph{The stochastic
  {Brusselator}: parametric noise destroys {Hopf} bifurcation}, Springer, New
  York, NY, 2015.

\bibitem{BlumenthalEngelNeamtu}
A.~Blumenthal, M.~Engel, and A.~Neamtu, \emph{On the pitchfork bifurcation for
  the {Chafee-Infante} equation with additive noise}, arXiv:2108.11073 (2021).

\bibitem{Callawayetal}
M.~Callaway, T.~S. Doan, J.~S.~W. Lamb, and M.~Rasmussen, \emph{The dichotomy
  spectrum for random dynamical systems and pitchfork bifurcations with
  additive noise}, Ann. Inst. Henri Poincar\'{e} Probab. Stat. \textbf{53}
  (2017), no.~4, 1548--1574. \MR{3729628}

\bibitem{ChemnitzEngel2022}
D.~Chemnitz and M.~Engel, \emph{Positive {L}yapunov exponent in the {H}opf
  normal form with additive noise}, arXiv:2212.06547 (2022), To appear in
  Communications in Mathematical Physics.

\bibitem{Collet13}
P.~Collet, S.~Mart\'inez, and J.~San~Mart\'in, \emph{Quasi-stationary
  distributions. {M}arkov chains, diffusions and dynamical systems}, Springer
  Berlin, Heidelberg, 2012.

\bibitem{CrauelFlandoli98}
H.~Crauel and F.~Flandoli, \emph{Additive noise destroys a pitchfork
  bifurcation}, J. Dyn. Differ. Equ. \textbf{10} (1998), 259–--274.

\bibitem{CrauelKloeden15}
H.~Crauel and P.~E. Kloeden, \emph{Nonautonomous and random attractors},
  Jahresbericht der Deutschen Mathematiker-Vereinigung \textbf{117} (2015),
  173--206.

\bibitem{Dey11}
S.~Dey, D.~Das, and P.~Parmananda, \emph{Intrinsic noise induced resonance in
  presence of sub-threshold signal in {B}russelator}, Chaos: An
  Interdisciplinary Journal of Nonlinear Science \textbf{21} (2011), no.~3.

\bibitem{Doanetal}
T.S. Doan, M.~Engel, J.S.W. Lamb, and M.~Rasmussen, \emph{{H}opf bifurcation
  with additive noise}, Nonlinearity \textbf{31} (2018), no.~10, 4567--4601.

\bibitem{EngelLambRasmussen1}
M.~Engel, J.S.W. Lamb, and M.~Rasmussen, \emph{Bifurcation analysis of a
  stochastically driven limit cycle}, Comm. Math. Phys. \textbf{365} (2019),
  no.~3, 935--942.

\bibitem{EngelLambRasmussen2}
\bysame, \emph{Conditioned {L}yapunov exponents for random dynamical systems},
  Trans. Am. Math. Soc. \textbf{372} (2019), no.~9, 6343--6370.

\bibitem{EngelOlicon23}
M.~Engel and G.~Olic\'on-M\'endez, \emph{A singular perturbation analysis of
  the {Brusselator}}, in preparation (2023).

\bibitem{EngelOlicon22}
M.~Engel, G.~Olic\'on-M\'endez, N.~Unger, and S.~Winkelmann,
  \emph{Synchronization and random attractors for reaction jump processes},
  arXiv:2207.00602 (2022).

\bibitem{gillespie1977exact}
D.~T. Gillespie, \emph{Exact stochastic simulation of coupled chemical
  reactions}, J. Phys. Chem. \textbf{81} (1977), no.~25, 2340--2361.

\bibitem{Szmolyan09}
I.~Gucwa and P.~Szmolyan, \emph{Geometric singular perturbation analysis of an
  autocatalator}, Discrete Contin. Dyn. Syst.--S \textbf{2} (2009), 783--806.

\bibitem{ImkellerLederer1999}
P.~Imkeller and C.~Lederer, \emph{An explicit description of the {L}yapunov
  exponents of the noisy damped harmonic oscillator}, Dynam. Stability Systems
  \textbf{14} (1999), no.~4, 385--405. \MR{1746113}

\bibitem{ImkellerLederer2001}
\bysame, \emph{Some formulas for {L}yapunov exponents and rotation numbers in
  two dimensions and the stability of the harmonic oscillator and the inverted
  pendulum}, Dyn. Syst. \textbf{16} (2001), no.~1, 29--61. \MR{1835906}

\bibitem{KloedenPlaten92}
P.~Kloeden and E.~Platen, \emph{Numerical solution of stochastic differential
  equations}, Springer Berlin, Heidelberg, 1992.

\bibitem{Kuehn14}
C.~Kuehn, \emph{Normal hyperbolicity and unbounded critical manifolds},
  Nonlinearity \textbf{27} (2014), 1351–--1366.

\bibitem{Kuehn15}
\bysame, \emph{Multiple time scale dynamics}, Springer Cham, 2015.

\bibitem{LeiteWilliams2019}
S.~C. Leite and R.J. Williams, \emph{A constrained {L}angevin approximation for
  chemical reaction networks}, Ann. Appl. Probab. \textbf{29} (2019), no.~3,
  1541–1608.

\bibitem{LiHouShen16}
X.~Li, J.~Hou, and Y.~Shen, \emph{Slow-fast effect and generation mechanism of
  {Brusselator} based on coordinate transformation}, Open Phys. \textbf{14}
  (2016), 261–--26.

\bibitem{LinYoung2008}
K.~K. Lin and L.~S. Young, \emph{Shear-induced chaos}, Nonlinearity \textbf{21}
  (2008), no.~5, 899--922. \MR{2412320}

\bibitem{Young13}
K.~Lu, Q.~Wang, and L.S. Young, \emph{Strange attractors for periodically
  forced parabolic equations}, vol. 224, Mem. Am. Math. Soc., 2013.

\bibitem{Newman2020}
Julian Newman, \emph{Synchronisation of almost all trajectories of a random
  dynamical system}, Discrete Contin. Dyn. Syst. \textbf{40} (2020), no.~7,
  4163--4177. \MR{4097539}

\bibitem{Prigogine68}
I.~Prigogine and R.~Lefever, \emph{Symmetry breaking instabilities in
  dissipative systems. {II}.}, J. Chem. Phys. \textbf{48} (1968), 1665--1700.

\bibitem{NicolisPrigogine}
I.~Prigogine and G.~Nicolis, \emph{Self-organisation in nonequilibrium
  systems}, Wiley, 1977.

\bibitem{wilkieWong08}
J.~Wilkie and Y.~M. Wong, \emph{Positivity preserving chemical {L}angevin
  equations}, Chem. Phys. \textbf{353} (2008), 132--138.

\bibitem{winkelmann2017hybrid}
S.~Winkelmann and C.~Sch{\"u}tte, \emph{Hybrid models for chemical reaction
  networks: Multiscale theory and application to gene regulatory systems}, J.
  Chem. Phys. \textbf{147} (2017), no.~11, 114115.

\bibitem{winkelmann2020}
\bysame, \emph{Stochastic dynamics in computational biology}, Springer, 2020.

\end{thebibliography}
\bibliographystyle{amsplain}

\end{document}